\documentclass[12pt]{amsart}

\usepackage{physics}
\usepackage{amssymb}
\usepackage{graphicx}
\usepackage{mathrsfs}
\usepackage{amsthm}
\usepackage{enumitem}
\usepackage{tikz-cd}
\usepackage[style=alphabetic]{biblatex}
\usepackage{mathtools}

\numberwithin{equation}{section}

\newtheoremstyle{fancy1}{10pt}{10pt}{\itshape}{12pt}{\textsc\bgroup}{.\egroup}{8pt}{ }
\newtheoremstyle{fancy2}{10pt}{10pt}{}{12pt}{\itshape}{.}{8pt}{ }

\theoremstyle{fancy1}
\newtheorem{thm}{Theorem}[section]
\newtheorem{cor}[thm]{Corollary}
\newtheorem{prop}[thm]{Proposition}
\newtheorem{lem}[thm]{Lemma}

\newtheorem{main}{Theorem}

\newtheorem*{cor*}{Corollary}

\theoremstyle{fancy2}
\newtheorem{defn}[thm]{Definition}

\theoremstyle{remark}
\newtheorem{rem}[thm]{Remark}
\newtheorem{rems}[thm]{Remarks}

\newcommand{\R}{\mathbb{R}}

\newcommand{\Z}{\mathbb{Z}}

\newcommand{\mb}[1]{\mathbb{#1}}

\usepackage[pdfusetitle]{hyperref}

\newcommand{\vv}[1]{v_{#1}}
\newcommand{\uu}{u}
\newcommand{\bfun}{\eta}

\newcommand{\scal}{\mathrm{Scal}}
\newcommand{\orthvect}{e}
\newcommand{\kgam}{k_{\gamma}}
\newcommand{\hess}{\mathcal{H}}

\newcommand{\mixedfactor}[1]{g_{f_{#1} f_{#1+1}}}
\newcommand{\mixedfactorF}[1]{F_{#1, #1+1}}

\newcommand{\leaf}[1]{\mathcal{F}_{p_{#1}}}
\newcommand{\leaff}{\mathcal{F}}

\begin{document}

\title{On the geometry of conullity two manifolds}
\author{Jacob Van Hook}

\begin{abstract}
We consider complete locally irreducible conullity two Riemannian manifolds with constant scalar curvature along nullity geodesics. 
There exists a naturally defined open dense subset on which we describe the metric in terms of several functions which are uniquely determined up to isometry.
In addition, we show that the fundamental group is either trivial or infinite cyclic. 
\end{abstract}
\maketitle



A manifold \(M^{n+2}\) has \emph{conullity two} if at every point \(p \in M\) we have a codimension two nullity distribution 
\begin{equation*}
	\Gamma_p \coloneqq \{T \in T_p M : R(T, X, Y, Z) = 0 \text{ for any }X, Y, Z \in T_p M\}.
\end{equation*} 
In other words, \(\Gamma_p\) is the kernel of the curvature tensor. 
A vector \(T \in \Gamma_p\) is called a \emph{nullity vector} and a geodesic \( \gamma \) is called a \emph{nullity geodesic} if \( \gamma'(t) \) is a nullity vector for all \( t \).
It is well known that if \( M \) is complete then the nullity geodesics are complete as well.

In \cite{Brooks:2021}, Brooks investigated complete 3-manifolds with constant Ricci eigenvalues \( (\lambda, \lambda, 0) \).
Equivalently, these manifolds are conullity 2 and have constant scalar curvature. 
We apply his techniques to study complete conullity two \((n+2)\)-manifolds \(M\) and relax the assumption of constant scalar curvature and instead only assume that it is constant along nullity geodesics. 

A key tool we use is the \emph{splitting tensor} \( C \colon \Gamma \times \Gamma^\perp \to \Gamma^\perp \) 
\begin{equation*}
	C(T, X) = C_T X \coloneqq -\left( \nabla_X T \right)^{\Gamma^\perp}.
\end{equation*}
The tensor \( C_T \) is nilpotent for every \( T \) if and only if the scalar curvature is constant along nullity geodesics, and if this is the case, then the open subset where \( C \neq 0 \), henceforth written as \( M_C \), is foliated by complete flat hypersurfaces which extends to a Lipschitz foliation \( \mathcal{F} \) on the closure of \( M_C \). 
Hence through every point of \( \bar{M}_C \) there is a \( C^{1,1} \) arc length parametrized curve \( \gamma \) orthogonal to the foliation and we consider the open subset \( M_F \subset M_C \) where \( \gamma \) has a smooth fully-defined Frenet frame, see definition \ref{defn:mf}.
If we assume that the metric is locally irreducible, then \( M_F \) is an open dense subset of \( M \), i.e. \( M = \bar{M}_C = \bar{M}_F \).
We then have the following result about the metric on connected components of \( M_F \):
\begin{main}\label{thm:themetriconaV}
	Suppose that \(M^{n+2}\) is a complete, simply connected, locally irreducible manifold with conullity 2 and whose scalar curvature is constant along nullity geodesics.
	Then on a connected component \(V\) of \(M_F\) there exist smooth coordinates
	 \begin{equation*}
	 	(x,\uu,\vv{1}, \dots, \vv{n}) \in (c_1, c_2) \times \R^{n+1}
	 \end{equation*}
	 and the metric has the form
	\begin{equation}\label{eq:bmetric}
		g_{\bfun, f_i} = \bfun(x,u)^2 \dd x^2 + \sum_{j=0}^{n} \bigg(\dd \vv{j} + \Big( \vv{j-1} f_{j}(x) - \vv{j+1} f_{j+1}(x) \Big) \dd x\bigg)^2 \tag{\( \star \)}
	\end{equation}
	for some functions \(f_1, \dots, f_{n} \colon (c_1, c_2) \to \R\) and \(\bfun \colon (c_1, c_2) \times \R \to \R\).
\end{main}

For simplicity in notation we assume in \eqref{eq:bmetric} that \(\vv{0} = \uu\), \(\vv{-1} = \vv{n+1} = 0\), and \(f_0 = f_{n+1} = 0\). 
We will use this convention throughout the paper.
Furthermore, \( f_1(x) = \abs{C_T(\gamma(x))} \) (up to sign) and \( \bfun_{\uu \uu} + \scal \cdot \bfun = 0 \) with \( \bfun(x,0) = 1 \) and \( \bfun_x(x,0) = \kgam(x) \), where \( \kgam \) is the geodesic curvature of \( \gamma(x) = (x,0,\dots, 0) \).
The leaves of the foliation \( \mathcal{F} \) are given by the level sets of \( x \).
The functions \( \bfun \) and \( f_i \) are, after a choice of base point, isometry invariants of the metric. 

We now discuss a method to construct smooth examples where \( M_F \neq M \) and where the foliation \( \mathcal{F} \) is smooth only on an open dense subset.
For this, start with a complete simply connected surface \( \Sigma \) with \( \scal \leq -2 \). 
Given a Lipschitz function \( H \) with Lipschitz constant 1, we will show that there exists a curve \( \gamma: \R \to \Sigma \) which is \( C^{1,1} \) such that the turning angle of \( \gamma' \) is equal to \( H \).
Thus the geodesic curvature of \( \gamma \) is \( H' \), defined almost everywhere.
We will then show that the geodesics orthogonal to \( \gamma \) foliate \( \Sigma \) and hence \( (x,\uu) \mapsto \exp_{\gamma(x)}(\uu X) \) defines Lipschitz coordinates on \( \Sigma \) where \( X \) is a unit vector orthogonal to \( \gamma' \). 
In these coordinates the metric has the form \( \bfun(x,\uu)^2 \dd{x}^2 + \dd{\uu}^2 \). 

We may then ask for conditions on the functions \( f_i \) and \( \bfun \) such that the resulting metric is guaranteed to be smooth on \( \bar{M}_F = M \). 
A sufficient condition is given by the following result:
\begin{main}\label{thm:smoothnessresultinintro}
	Let \( \Sigma \) be a complete, simply connected surface with \( \scal \leq -2 \). Let \(f_1, \dots, f_n \colon \R \to \R\) be smooth functions, and \(H \colon \R \to \R\) Lipschitz with Lipschitz constant 1 which is smooth on an open dense set \(S \subset \R\).
	Define a metric on \( M = \Sigma \times \R^n = \R^{n+2} \) which has the form \eqref{eq:bmetric} on \( (S \times \R) \times \R^n \).
	If, as \( x \to \R \setminus S \), each \( f_1, \dots, f_n \) vanishes to infinite order and every product 
	\begin{equation}\label{eq:smoothnesscondition}
		f_i^{(k)}(x) \prod_{j} \frac{\partial^{a_j + b_j} \bfun}{\partial x^{a_j} \partial \uu^{b_j}}(x, \uu) \to 0, \quad \uu \in \R, \, k \geq 0, \, a_j + b_j > 0, \tag{\( \dagger \)} 
	\end{equation}
	does as well, then the metric extends smoothly to all of \( M \).
	Furthermore, the metric is complete with conullity 2 and is locally irreducible if no \( f_i \) vanishes on an open set. 
\end{main}

Here the foliation \( \mathcal{F} \) is smooth on \( (S \times \R) \times \R^n \), but only Lipschitz on \( M \) since \( \bfun(x,\uu) \) is smooth in the \( \uu \) variable but only Lipschitz in the \( x \) variable. 
In particular, \( \bfun \) and its derivatives are not necessarily continuous or even defined for \( x \in \R \setminus S \).
Using Theorem \ref{thm:smoothnessresultinintro}, one easily constructs examples where \( S \) is the complement of a Cantor set in \( \R \).

The strategy of our proofs uses the methods, and some of the results, in \cite{Brooks:2021}. 
The splitting tensor \( C \) defines a basis \( e_1 \) \( e_2 \) which is Lipschitz on \( M \).
In dimension 3 this defines a Lipschitz basis \( e_1 \), \( e_2 \), \( e_3 \) and the coordinates in \eqref{eq:bmetric} are Lipschitz on all of \( M \). 
This is no longer true if \( \dim M > 3 \).
We first find a smooth basis on \( M_F \) which extends \( e_1 \) and \( e_2 \) and we denote these new vectors by \( T_1, \dots, T_n \). 
For \( p \in M_F \) they are defined as the Frenet frame of the curve \( \gamma \) orthogonal to \( \mathcal{F} \) with \( \gamma(0) = p \).
In other words \( \nabla_{\gamma'} T_i = -a_i T_{i-1} + a_{i+1} T_{i+1} \) with \( a_i > 0 \).
This requires that \( \gamma \) is sufficiently regular.
\( M_F \) is the set of points where all \( a_i \neq 0 \).
However, for \( n > 3 \), the frame is not necessarily continuous. 
In fact, one finds such examples using Theorem \ref{thm:smoothnessresultinintro}.
We then derive a number of properties of a metric of the form \eqref{eq:bmetric} and use the exponential map to find an isometry from it to any conullity two manifold. 

If \( M \) is complete, we will see that \( \scal < 0 \). 
A special case is when \(\scal_M = -2 \), which is equivalent to assuming that \( M \) is curvature homogeneous.
This implies that \( \bfun(x,\uu) = \cosh(\uu) + h(x)\sinh(\uu) \) with \( h(x) = \kgam(x) \) and both Theorem \ref{thm:themetriconaV} and Theorem \ref{thm:smoothnessresultinintro} are then a generalization of the results in \cite{Brooks:2021} to higher dimensions.
If the metric is analytic then \( M_F = M \) and hence Theorem \ref{thm:themetriconaV} is a classification of all analytic conullity 2 metrics with \( \scal \) constant along nullity geodesics. 

Since conullity 2 and \( \scal < 0 \) implies \( \sec \leq 0 \), any simply connected example is diffeomorphic to \( \R^{n+2} \).
To study the topology in general we thus are interested in the fundamental group. 
Focusing on the locally irreducible case, we have the following result:
\begin{main}\label{thm:fundgroup}
	Let \(M\) be a complete conullity 2 manifold which is locally irreducible.
	If, furthermore, \( \scal\) is constant along nullity geodesics, then the fundamental group of \(M\) is either trivial or \(\Z\).
\end{main}
One may construct examples where \( \pi_1(M) = \Z \) by choosing the \( f_i(x) \) and \( \bfun(x,\uu) \) to be nonvanishing and periodic in \( x \) with the same period.
One may also generalize the results in \cite{Brooks:2021} if \( M \neq \bar{M}_C \), i.e. the metric is allowed to be locally reducible on regions of \( M \), but we will not pursue this direction here.

One easily sees that the splitting tensor \( C_T \) is either nilpotent or has complex eigenvalues. 
The former case is studied in this paper.
Conullity two manifolds with \( C_T \) having complex eigenvalues seem to be more complicated.
In \cite{Szabo:1985} it was shown that all such locally irreducible and complete examples must be three dimensional and he constructed many local examples.
To our knowledge, the only known complete examples are submanifolds in Euclidean space, see \cite{Takagi:1972}, \cite{Dajczer:2001}, and \cite{Boeckx:1996}.

An outline of the paper is as follows.
We begin with some preliminary results in section \ref{sec:prelim}, and in section \ref{sec:remainingnullity} we construct a special nullity frame.
In section \ref{sec:metric} we prove properties of the metric \eqref{eq:bmetric} which leads us to prove Theorem \ref{thm:themetriconaV} in section \ref{sec:onmanifold}.
We then investigate foliations on surfaces in section \ref{sec:foliations} and use the results there to prove Theorem \ref{thm:smoothnessresultinintro}.
Finally in section \ref{sec:fungroup} we investigate the topological properties of these manifolds and prove Theorem \ref{thm:fundgroup}.

I would like to thank my Ph.D. thesis advisor, Dr. Wolfgang Ziller, for his guidance and support throughout every step of this project.


\section{Preliminaries}\label{sec:prelim}
In this section we summarize some properties of our metrics by generalizing and adapting the proofs of \cite{Brooks:2021} to dimension \( n > 3 \).

One should note that the curvature tensor of a conullity two manifold is determined by the scalar curvature. 
Indeed, at each point the curvature tensor is like that of a product \(\Sigma \times \R^n\) where \(\Sigma\) is a surface.
If \( M \) is complete we will show that \( \scal < 0 \).

The nullity distribution \(\Gamma\) is everywhere \(n\)-dimensional by assumption.
It is known to be completely integrable with totally geodesic flat leaves which are complete if \( M \) is (see \cite{Maltz:1972}).

Recall the \emph{splitting tensor} \(C \colon \Gamma \times \Gamma^\perp \to \Gamma^\perp\) is defined as
\begin{equation}
C(T, X) = C_T X \coloneqq -\left( \nabla_X T \right)^{\Gamma^\perp}.
\end{equation}
We begin with some facts about this tensor, see \cite{Florit:2016} and \cite{Brooks:2021}.

Fix some nullity geodesic \(\gamma(t)\). 
If \(T = \gamma'\), then in some parallel basis along \( \gamma \), \( C_T \) satisfies the Riccati equation \({C_T}'(t) = C_T^2(t)\) whose solutions are \(C_T(t) = C_0(I - t C_0)^{-1}\) for some initial condition \(C_0 = C_T(0)\). 
Any real eigenvalues of \(C_T\) must therefore be 0 as otherwise the splitting tensor would blow up in finite time. 
As \(C_T\) is a linear map between two dimensional vector spaces, it must be nilpotent or it has two complex eigenvalues.

The second Bianchi identity implies that
\begin{equation*}\label{eq:scalprime}
\scal' = \tr(C_T) \scal.
\end{equation*}
along a nullity geodesic. 
If the scalar curvature is constant along nullity geodesics we have \(\tr(C_T) = 0\). 
Furthermore, \((\tr C_T)' = \tr C_T^2 = (\tr C_T)^2 - 2 \det C_T \) and therefore \(\det C_T = 0\) as well, so \( C_T\) is nilpotent.
Conversely, if \( C_T \) is nilpotent, then \( \scal \) is constant along the nullity geodesic. 
From now on we will assume that \( C_T \) is nilpotent. 

The space of self-adjoint \(2 \times 2\) matrices is a 3-dimensional subspace of the 4-dimensional space of \( 2 \times 2 \) matrices. 
This subspace intersects the space of nilpotent matrices trivially and so the image of \(C: T \mapsto C_T\) is 1-dimensional. 
Therefore if \(C \neq 0\) then at each point \(p \in M\) there is a unique (up to sign) \(T_1 \in \Gamma\) such that \(C_{T_1} \neq 0\), and for any nullity vector \(T \perp T_1\), \(C_T = 0\).
Since \(C_{T_1}\) is nilpotent, there exists an orthonormal basis \(\{e_1, e_2\}\) of \(\Gamma^\perp\), parallel in the \( T_1 \) direction and unique up to sign, such that with respect to that basis \( C_{T_1} \) has the form
\begin{equation*}
C_{T_1} = \begin{pmatrix} 0 & a_1 \\ 0 & 0 \end{pmatrix}
\end{equation*}
for some smooth function \(a_1\).
That is,
\begin{equation*}
\left(\nabla_{e_1} T_1 \right)^{\Gamma^\perp} = 0, \quad \left(\nabla_{e_2} T_1 \right)^{\Gamma^\perp} = -a_1 e_1, \qand (\nabla_{e_i} T)^{\Gamma^\perp} = 0 
\end{equation*}
for all \( T \in \Gamma \) orthogonal to \( T_1 \). 
Furthermore, since \({C_{T_1}}' = C_{T_1}^2 = 0\), it follows that \(T_1(a_1) = 0\).

\begin{defn}
Let \(M_C \subset M\) be the set of points on which \(C\) (equivalently \(a_1\)) is nonzero.
\end{defn}

We first investigate properties of the splitting tensor on \(M_C\), so for the rest of this section suppose that \(a_1 \neq 0\). 
There exist smooth functions \(\alpha\) and \(\beta\) such that: 
\begin{align*}
\nabla_{e_1}e_1 &= \alpha e_2; \quad \nabla_{e_1}e_2 = -\alpha e_1; \\ 
\nabla_{e_2} e_1 &= a_1 T_1 - \beta e_2; \quad \nabla_{e_2} e_2 = \beta e_1, \nonumber
\end{align*}
which follow from the definition of \( T_1 \) and the parallel orthonormal frame \( \{e_1, e_2\} \) along the nullity geodesic.  

\begin{prop}\label{prop:initialcurvcomps}
Using the above notation, on \(M_C\) we have:
\begin{enumerate}[label = (\alph*)]
\item
\(\alpha = 0\), so \(\nabla_{e_1} e_1 = 0\);
\item
\(T(a_1) = T(\beta) = 0\) and \(\nabla_T T_1 = \nabla_T e_1 = \nabla_T e_2 = 0\) for any nullity vector \(T \in \Gamma\);
\item
\(\nabla_{e_1} T_1 = 0\), \(e_1(a_1) = a_1 \beta\), and \(e_1(\beta) - \beta^2 = R(e_1, e_2, e_2, e_1) = \frac{1}{2}\scal\).
\end{enumerate}
\end{prop}

\begin{proof}
A computation shows that
\begin{equation*}\label{eq:curvcomp122}
R(e_1, e_2) e_2 = \alpha a_1 T_1 + (e_1(\beta) + e_2(\alpha) - \alpha^2 - \beta^2) e_1.
\end{equation*}
Since \( T_1 \in \Gamma \) and \(a_1 \neq 0\), we have that \(\alpha = 0.\) Thus
\begin{equation*}\label{eq:e1beta}
e_1(\beta) - \beta^2 = R(e_1, e_2, e_2, e_1).
\end{equation*}
and \(\nabla_{e_1} e_1 = 0.\)
Furthermore,
\begin{equation*}\label{eq:curvcomp211}
R(e_2, e_1)e_1 = (a_1 \beta - e_1(a_1))T_1 - a_1 \nabla_{e_1}T_1 + (e_1(\beta)-\beta^2)e_2.
\end{equation*}
Notice that \(\nabla_{e_1}T_1 \in \Gamma\) and is orthogonal to \(T_1\), so this tells us that \(a_1 \beta = e_1(a_1)\) and \(\nabla_{e_1} T_1 = 0.\)

To verify the claims involving an arbitrary nullity vector we check separately the cases where \(T\) is orthogonal to \(T_1\) and where \(T = T_1\) which is sufficient by linearity.
Let \(T \in \Gamma\) be unit-length and orthogonal to \(T_1\). 
Notice that \(\Gamma\) being totally geodesic implies that \(\nabla_T e_2 \in \Gamma^\perp\), so let this be \(f_1 e_1\). 
Likewise, let \(\nabla_T e_1 = -f_1 e_2\).
We also know that \(\nabla_{e_2}T \in \Gamma\) so let this be \(f_2 \tilde{T}\) where \(\tilde{T} \in \Gamma\) is unit-length.

A computation shows that
\begin{equation*}\label{eq:curvcompT21}
R(T, e_2)e_1 = T(a_1)T_1 + a_1 \nabla_T T_1 + (e_2(f_1) - T(\beta)) e_2 + f_2 \nabla_{\tilde{T}} e_1.
\end{equation*}
which implies \(T(a_1) = 0\) and \(\nabla_T T_1 = 0.\)
Next we have
\begin{equation*}\label{eq:curvcomp1TT1}
R(e_1, T)T_1 = -f_1 \nabla_{e_2} T_1
\end{equation*}
and as \(\nabla_{e_2}T_1 \neq 0\) as it has an \(e_1\)-component, \(f_1 = 0\) so we conclude that \(\nabla_T e_1 = \nabla_T e_2 = 0\) and we also get that \(T(\beta) = 0.\)
This concludes the case of \( T \) orthogonal to \( T_1 \).
For the case of \( T = T_1 \) we have \( \nabla_{T_1} e_i = 0 \) by construction.
Since
\begin{equation*}\label{eq:curvcompt121}
R(T_1, e_2)e_1 = T_1(a_1) T_1 + a_1 \nabla_{T_1} T_1 - T_1(\beta) e_2,
\end{equation*}
and \(\nabla_{T_1} T_1 \in \Gamma\) is orthogonal to \(T_1\), it follows that \(T_1(a_1) = T_1(\beta) = 0\) and \(\nabla_{T_1} T_1 = 0.\)
\end{proof}

In summary, we showed that all of \(e_1, e_2, T_1, a_1\), and \(\beta\) are constant along nullity directions.
Furthermore, \(e_1, e_2\), and \(T_1\) are also parallel in the \(e_1\) direction.

\subsection{Foliation by Flat Hyperplanes}
With the help of the splitting tensor we define a distribution on \(M_C\) which will foliate our manifold with complete flat hypersurfaces. 

\begin{defn}
Let \( D \) be the distribution on \(M_C\) which is spanned by \(e_1\) and \(\Gamma\).
\end{defn}

\begin{prop}
\begin{enumerate}[label = (\alph*)]
	\item The distribution \(D\) is integrable with totally geodesic flat leaves. Henceforth let \(\mathcal{F}_p\) denote the leaf of \(D\) containing \(p \in M_C\). \label{prop:integrableflatleaves}
	\item If \( M \) is complete, then \( \scal < 0 \). Furthermore, for \( p \in M_C \), the leaves \(\mathcal{F}_p\) are complete and contained in \(M_C\). \label{prop:negativescal}
\end{enumerate}
\end{prop}

\begin{proof}
From Proposition \ref{prop:initialcurvcomps} and the fact that \( \Gamma \) is totally geodesic, it follows that \( D \) is completely integrable with totally geodesic leaves. 
The leaves are flat as they contain only one non-nullity direction, so \ref{prop:integrableflatleaves} is proved.

For \ref{prop:negativescal},  we will show that geodesics with initial velocity in \(D\) are complete and remain in \(M_C\).
Starting at \(p \in M_C\), \(\nabla_{e_1}e_1 = 0\), so we have that the integral curve of \(e_1\) is a geodesic \(\eta\) defined as long as \(C \neq 0\). 
Since \(M\) is complete, we know that \(\eta\) is a complete geodesic so we just have to show it remains in \(M_C\).
We now use the facts that \(\alpha = 0\), \(e_1(\beta)-\beta^2 = R(e_1, e_2, e_2, e_1) = \sec(e_1, e_2)\), and \(e_1(a_1) = a_1 \beta\).
Letting a prime denote an \(e_1\) derivative we have
\begin{equation*}\label{eq:1overa1jacobi}
	\left( \frac{1}{a_1}\right)'' = -\left(\frac{{a_1}'}{a_1^2}\right)' = -\left(\frac{\beta}{a_1}\right)' = -\frac{\sec(e_1,e_2)}{a_1}
\end{equation*}
and therefore
\begin{equation}\label{eq:1overa1}
	\left(\frac{1}{a_1}\right)'' + \frac{1}{a_1}\sec(e_1,e_2) = 0.
\end{equation}

Thus \(\frac{1}{a_1}\) satisfies the Jacobi equation of a Jacobi field along \(\eta\) in the \(e_2\) direction. 
By \eqref{eq:1overa1}, if \( \scal > 0 \) then \(\frac{1}{a_1} \) has a zero in finite time, contradicting completeness. 
On the other hand, if \( \scal < 0 \) then \eqref{eq:1overa1} shows that \( a_1 \) will be nonzero along \( \eta \) so it remains in \( M_C \).
As \( T(a_1) = 0 \) for any nullity vector \( T \), the whole leaf must be contained in \( M_C \). 
\end{proof}

Since the distribution \(D\) is integrable we get a foliation \(\mathcal{F}\) of \(M_C\) with complete flat totally geodesic leaves. 
This extends continuously to a foliation on the closure of \(M_C\) with leaves that are complete, flat, totally geodesic hypersurfaces.
Since any codimension 1 geodesic foliation is locally Lipschitz \cite{Zeghib:1999}, there exists a unique \( C^{1,1} \) curve \( \gamma \) orthogonal to the foliation through each point \( p \in M_C \) which can thus be viewed as an integral curve of \( e_2 \). 
This curve satisfies the following properties: 
\begin{prop}\label{prop:curvehitsonce}
	Suppose that \( M \) is a complete, simply connected, conullity 2 manifold whose scalar curvature is constant along nullity geodesics. 
	For any \( p \in \bar{M}_C \), the closure of \( M_C \), there exists a unique, maximal (in \( \bar{M}_C \)) \( C^{1,1} \) integral curve \( \gamma \) of \( e_2 \) which is orthogonal to \( \mathcal{F} \) at every point. 
	Furthermore, \( \gamma \) intersects each leaf of \( \mathcal{F} \) that is in the connected component of \( \bar{M}_C \) containing \( p \) exactly once.
\end{prop}
This is proved by following the argument of Proposition 16 in \cite{Brooks:2021} where he proves the same result in the dimension 3 curvature homogeneous case. 
In particular, if \( M \) is locally irreducible, as assumed in Theorem \ref{thm:themetriconaV}, then \( \gamma\) is globally defined and meets each leaf of \( \mathcal{F} \) exactly once.


\section{Defining the remaining canonical nullity vectors}\label{sec:remainingnullity}
We now define an orthonormal basis of \(D\) at each point by completing \(T_1\) to an orthonormal basis of \(\Gamma\).
Start by fixing a point \(p\). 
Let \(\gamma\) be a maximal integral curve of \(e_2\) through \(p\) as above.
Then we define \(T_2, \dots, T_{n}\) to be the Frenet frame along \(\gamma\) whenever it is well-defined.
In other words, let
\begin{equation*}
	X_1 := (\nabla_{e_2}T_1)^{\text{span}\{e_1, e_2, T_1\}^\perp} \qand a_2 = \abs{X_1}
\end{equation*}
where we are projecting onto the orthogonal complement of the span of \(e_1, e_2\) and \(T_1\).
Then, if \(a_2 \neq 0\), we let \(T_2 = \frac{X_1}{a_2}\) and
\begin{equation*}
	\nabla_{e_2}T_1 = -a_1 e_1 + a_2 T_2. 
\end{equation*}
We repeat this process, considering \(\nabla_{e_2} T_2\), projecting it to an orthogonal complement and normalizing it. 
This gives
\begin{equation*}
	\nabla_{e_2} T_i = -a_i T_{i-1} + a_{i+1} T_{i+1}
\end{equation*}
for \(i = 2, \dots, n-1\) and lastly
\begin{equation*}
	\nabla_{e_2}T_{n} = -a_{n}T_{n-1}.
\end{equation*}
Thus, \(\{e_2, e_1, T_1, \dots, T_{n}\}\) as described above is the Frenet frame of \(\gamma\) when \(\gamma\) is sufficiently regular.
The curvature of \(\gamma\) is \(\kappa = \beta\), the torsion is \( \tau = a_1\), and the generalized curvatures are \(\chi_{i+1} = a_i\). 
Notice that this frame is well-defined up to sign and hence are isometry invariants of \( M \) after choosing a base point \( \gamma(0) \).

\begin{defn}\label{defn:mf}
	Let \(M_F \subset M_C\) be the open subset where \(a_i \neq 0\) for \(i = 1, \dots, n\).
	Equivalently, this is the set on which the full Frenet frame of \( \gamma\) is well-defined.
\end{defn}

\begin{prop}\label{prop:niceremainingT}
	On \(M_F\) we have the following: 
	\begin{equation*}
		e_1(a_i) = a_i \beta, \quad \nabla_{e_1}T_i = 0, \quad T(a_i) = 0, \qand \nabla_{T}T_i = 0,
	\end{equation*}
	where \(T \in \Gamma\) and \(1 \leq i \leq n\).
	In particular, on each leaf of the foliation \( \mathcal{F} \) each \( a_i \) is either zero or nonzero everywhere. Thus \( \mathcal{F} \) restricts to a complete foliation on \( M_F \) which we will also call \( \mathcal{F} \). 
\end{prop}

\begin{proof}
	The proof is by induction. 
	Recall that Proposition \ref{prop:initialcurvcomps} covers the \(i = 1\) case, as well as that \(\nabla_T e_1 = 0\) for any \(T \in \Gamma\).
	
	We check the \(i = 2\) case directly since the expression of \(\nabla_{e_2} T_1\) is slightly different than the rest.
	Observe that
	\begin{equation*}\label{eq:curvcomp12T1}
		R(e_1, e_2)T_1 = (e_1(a_2) - a_2 \beta) T_2 + a_2 \nabla_{e_1}T_2,
	\end{equation*}
	and hence we conclude that \(e_1(a_2) = a_2 \beta\) and \(\nabla_{e_1}T_2 = 0\).

	To complete the other half of the \(i=2\) case we first consider the case of \(T = T_1\).
	Then we have
	\begin{equation*}\label{eq:curvcompT12T1}
		R(T_1, e_2)T_1 = -T_1(a_1)e_1 + T_1(a_2)T_2 + a_2 \nabla_{T_1}T_2,
	\end{equation*}
	and hence we conclude that \(T_1(a_2) = 0\) and \(\nabla_{T_1}T_2 = 0\).
	
	It remains to check the case of \(T \perp T_1\).
	Then we have
	\begin{equation*}\label{eq:curvcompT2T1}
		R(T, e_2)T_1 = -T(a_1) e_1 + T(a_2) T_2 + a_2 \nabla_T T_2,
	\end{equation*}
	and hence \(\nabla_T T_2 = 0\) and \(T(a_2) = 0\).
	This shows that the equations holds for \(i =2\).
	
	We may now proceed with the induction argument in a very similar manner as above.
	Suppose the result holds up to \(T_i\). 
	Then we have
	\begin{equation*}\label{eq:curvcomp12Ti}
		R(e_1, e_2)T_i = (e_1(a_{i+1}) - a_{i+1}\beta)T_{i+1} + a_{i+1} \nabla_{e_1} T_{i+1},
	\end{equation*}
	and so \(e_1(a_{i+1}) = a_{i+1}\beta\) and \(\nabla_{e_1} T_{i+1} = 0\).
	
	For \(T = T_1\) we have
	\begin{equation*}\label{eq:curvcompT12Ti}
		R(T_1, e_2)T_i = T_1(a_{i+1})T_{i+1} + a_{i+1} \nabla_{T_1} T_{i+1},
	\end{equation*}
	and so \(T_1(a_{i+1}) = 0\) and \(\nabla_{T_1} T_{i+1} = 0\). 
	On the other hand, if \(T \perp T_1\), then
	\begin{equation*}\label{eq:curvcompT2Ti}
		R(T, e_2)T_i = T(a_{i+1})T_{i+1} + a_{i+1} \nabla_T T_{i+1},
	\end{equation*}
	and so \(T(a_{i+1}) = 0\) and \(\nabla_T T_{i+1} = 0\) for all \(T\).
	This finishes the proof.
\end{proof}

With respect to the basis \(\{e_2, e_1, T_1, \dots, T_{n}\}\), the matrix of \(\nabla_{e_2}\) is
\begin{equation*}
	\begin{pmatrix} 0 & -\beta & 0 & 0 & \dots & \dots & 0 \\
		\beta & 0 & -a_1 & 0 & \dots & \dots & 0 \\
		0 & a_1 & 0 & -a_2 & \dots & \dots & 0 \\
		0 & 0 & a_2 & 0 & \ddots &  & 0 \\
		\vdots & \vdots & \vdots & \ddots & \ddots & \ddots & \vdots \\
		\vdots & \vdots & \vdots &  & \ddots & 0 & -a_{n} \\
		0 & 0 & 0 & \dots & \dots & a_{n} & 0
	\end{pmatrix}
\end{equation*}
and \( \nabla_v = 0 \) for all \( v \perp e_2 \).


\section{Candidates for the metric}\label{sec:metric}
In this section investigate the properties of a certain metric with the goal of showing that this is the form of the metric on \( M_F \).

\subsection{Defining the metric}
Suppose we are given some positive function \(\bfun(x,\uu)\), normalized such that \(\bfun(x,0) = 1\) for all \(x\).
Furthermore suppose also that \(\bfun_{\uu \uu} \neq 0\). 
Then we may consider the metric of the form \eqref{eq:bmetric} in coordinates \((x, \uu, \vv{1}, \dots, \vv{n}) \in (c_1, c_2) \times \R^{n+1}\) defined by
\begin{equation*}
		g_{\bfun, f_i} = \bfun(x,u)^2 \dd x^2 + \sum_{j=0}^{n} \bigg(\dd \vv{j} + \Big( \vv{j-1} f_{j}(x) - \vv{j+1} f_{j+1}(x) \Big) \dd x\bigg)^2. \tag{\( \star \)}                                                                                                                           
\end{equation*}
Notice that we do not assume that the functions \( f_i \) are non-zero everywhere.

\begin{thm}\label{lem:bfunmetricprops}
If \(g_\bfun\) is the metric given by \eqref{eq:bmetric} where \(\bfun, f_i\) are smooth functions, with \(\bfun, \bfun_{\uu \uu} \neq 0\) and \( \bfun(x,0) = 1 \), then
\begin{enumerate}[label = (\alph*)]
\item \label{item:isametric}
\(g_\bfun\) is a metric of conullity 2 with scalar curvature \(\scal = -2\frac{\bfun_{\uu \uu}}{\bfun}\).

\item \label{item:nullity}
The coordinate vectors \(\pdv{v_i}\) are nullity vectors.

\item \label{item:orthonormal}
The basis \(\left\{\orthvect, \pdv{\uu}, \pdv{\vv{1}}, \dots, \pdv{\vv{n}} \right\}\) is orthonormal, where
\begin{equation*}
\orthvect = \frac{1}{\bfun(x,\uu)}\Bigg(\pdv{x} - \sum_{i = 0}^n (\vv{i-1} f_{i}(x) - \vv{i+1} f_{i+1}(x))\pdv{\vv{i}} \Bigg).
\end{equation*}

\item \label{item:a1def}
When \(f_1(x) \neq 0\), \(e_1 = \pdv{\uu}\), \(e_2 = \orthvect\), and \(T_1 = \pdv{\vv{1}}\) where \(e_1, e_2\), and \(T_1\) are defined by the splitting tensor. 
On the subset where all \( f_i(x) \neq 0 \), \(\left\{\orthvect, \pdv{\uu}, \pdv{\vv{1}}, \pdv{\vv{2}}, \dots, \pdv{\vv{n}} \right\} \) is the Frenet frame of \( \gamma(x) = (x,0, \dots, 0) \).

\item \label{item:generalaiformula}
The functions \( a_i \) and \( \beta \) from section \ref{sec:remainingnullity} are given by
\begin{equation*}
a_i(x, \uu, \vv{i}) = \frac{f_i(x)}{\bfun(x,\uu)} \quad \text{and} \quad \beta(x,\uu, \vv{i}) = -\frac{\bfun_{\uu}(x,\uu)}{\bfun(x,\uu)}.
\end{equation*}

\item \label{item:geodcurv}
The leaves \(\mathcal{F}_p\) are hyperplanes with \(x\) constant on each, and \( \gamma(x) \) is an arc length parametrized curve orthogonal to the leaves with geodesic curvature \(-\bfun_{\uu}(x,0)\).

\item \label{item:bfunlocirred}
The metric \(g_\bfun\) is locally irreducible if and only if each \(f_i^{-1}(0)\) contains no open subsets. 

\end{enumerate}
\end{thm}

\begin{proof}
Parts \ref{item:isametric}\textendash\ref{item:geodcurv} are straightforward computations and are left to the reader.

For \ref{item:bfunlocirred}, if some \(f_{i_0}(x)\) vanishes on an interval \((b_1, b_2)\) then on \((b_1, b_2) \times \R^{n+1}\) we have two parallel subdistributions \( D_1 \) and \( D_2 \) given by
\begin{equation*}
D_1 \coloneqq \text{span}\{e_2, e_1\} \qand D_2 \coloneqq \text{span}\left\{\pdv{\vv{1}}, \dots, \pdv{\vv{n}}\right\},
\end{equation*}
in the case of \( i_0 = 1 \) and by
\begin{equation*}
D_1 \coloneqq \text{span}\left\{e_2, e_1, \pdv{\vv{1}}, \dots, \pdv{\vv{i_0-1}}\right\} \qand D_2 \coloneqq \text{span}\left\{\pdv{\vv{i_0}}, \dots, \pdv{\vv{n}}\right\},
\end{equation*}
in the case of \( i_0 \geq 2 \). 
These are parallel since all derivatives are already zero except possibly those in the \( e_2 \) direction.
Since \( f_{i_0}(x) = 0 \) for \( x \in (b_1, b_2) \), part \ref{item:generalaiformula} implies that \( a_{i_0} = 0 \) on \( (b_1, b_2) \times \R^{n+1} \).
Thus \( \nabla_{e_2}X \in D_1 \) for all \( X \in D_1 \) since
\begin{equation*}
	\nabla_{e_2} \pdv{\vv{i_0-1}} = - a_{i_0-1} \pdv{\vv{i_0-2}} + a_{i_0} \pdv{\vv{i_0}}
\end{equation*}
and \( a_{i_0} = 0 \). 
Similarly, \( \nabla_{e_2}Y \in D_2 \) for all \( Y \in D_2 \).
By the de Rham decomposition theorem, the metric is locally reducible.

Now suppose that all \(f_i(x)\) do not vanish on any open interval. 
We will show that the tangent bundle contains no proper parallel subdistributions, which implies that the metric is locally irreducible. 

Suppose we do have some proper parallel subdistribution, call it \(Q\).
Suppose also that there exists an open interval \((b_1, b_2)\) on which \(Q\) is not orthogonal to both \(e_1\) and \(e_2\), so it has a section of the form
\begin{equation*}
q = p_1 e_1 + p_2 e_2 + q_i \pdv{\vv{i}}
\end{equation*}
for some functions \(p_1\),  \(p_2\), \(q_i\) with \(p_1^2 + p_2^2 \neq 0\). 
Since parallel distributions are closed under covariant derivatives they are also closed under the curvature tensor. 
In particular, we can apply \(R(e_1, e_2)\) to our distribution. 
Recall that
\begin{equation*}
R(e_1, e_2)e_1 = \frac{\bfun_{\uu \uu}}{\bfun} e_2, \quad
R(e_1, e_2)e_2 = -\frac{\bfun_{\uu \uu}}{\bfun} e_1, \qand
R(e_1, e_2)\pdv{\vv{i}} = 0.
\end{equation*}
Hence we get, after applying the curvature tensor twice, that both \(\frac{\bfun_{\uu \uu}}{\bfun} (p_1 e_2 - p_2 e_1)\) and \(-\frac{\bfun_{\uu \uu}}{\bfun} (p_1 e_1 + p_2 e_2)\) are sections of \(Q\).
Since \(p_1^2 + p_2^2 \neq 0\), both \(e_1\) and \(e_2\) are in \(Q_p\). 
As no \(f_i^{-1}(0)\) contains any open subset, on an open subset of \( (b_1, b_2) \) we may apply \(\nabla_{e_2}\) repeatedly to conclude that each \(\pdv{\vv{i}}\) is in \(Q\), so \(Q\) is not a proper subdistribution.

If no such open interval exists then \(Q\) is orthogonal to \(e_1\) and \(e_2\) on a dense set of points and thus everywhere by continuity. 
Then \( Q^\perp \) contains \( e_1 \) and \( e_2 \), and therefore the entire nullity distribution as well, so again this splitting is trivial. 
\end{proof}

\begin{rems}
	\begin{enumerate}[label = (\alph*)]
		\item The proof shows that if \( f_i(x) \equiv 0 \) for some \( i \), then \( M \) is an isometric product \( N^{i+1} \times \R^{n-i+1} \) where \( N^{i+1} \) is conullity two and with a flat metric on \( \R^{n-i+1} \).
		\item The condition in \ref{item:bfunlocirred} that no \(f_i^{-1}(0)\) contains an open subset is equivalent to the condition that \(\bigcup_{i} f_i^{-1}(0)\) contains no open subsets by the Baire category theorem.
		\item Notice that \( \bfun_{\uu \uu} \neq 0 \) is only really needed to conclude that \( g \) has conullity two and to ensure that \( e_1 \) and \( e_2 \) are well-defined. 
		By removing this assumption we may glue locally irreducible strips where \( \bfun_{\uu \uu} \neq 0 \) with flat strips where \( \bfun_{\uu \uu} = 0 \). 
	\end{enumerate}
\end{rems}

\subsection{Completeness of the metric}
\begin{lem} \label{lem:bfuncompleteness}
The manifold \((c_1, c_2) \times \R^{n+1}\) with the metric \eqref{eq:bmetric} is complete if \((c_1, c_2) = \R\) and there exists a positive continuous function \(r(\uu)\) such that \(\bfun(x,\uu) \geq r(\uu)\) for all \(x, \uu\). 
\end{lem}

\begin{proof}
We use a coordinate change from \cite{Kowalski:1992} that turns our metric into a warped product. 
Define the \((n+1) \times (n+1)\) matrix \(A(x)\) to be
\begin{equation*} 
	A(x) = \begin{pmatrix} 0 & -f_1(x) & 0 &  \cdots & \cdots & 0 \\
		f_1(x) & 0 & -f_2(x) &  \cdots & \cdots & 0 \\
		0 & f_2(x) & 0 &  \ddots &  & 0 \\
		\vdots & \vdots & \ddots &  \ddots & \ddots & \vdots \\
		\vdots & \vdots &  &  \ddots & 0 & -f_{n}(x) \\
		0 & 0 & 0 &  \cdots & f_{n} (x) & 0
	\end{pmatrix},
\end{equation*}
and then consider the matrix differential equation \(S'(x) = S(x)A(x)\) with initial condition \(S(0) = Id_{n+1}\).
Since \(A\) is skew symmetric and \((S^T S)(0) = Id\), it follows that \( S(x) \in SO(n+1) \).
Now consider the coordinate change
\begin{equation*}
\begin{pmatrix} \theta_0 \\ \theta_1 \\ \vdots \\ \theta_{n} \end{pmatrix} = S(x) \begin{pmatrix} \uu \\ \vv{1} \\ \vdots \\ \vv{n} \end{pmatrix}.
\end{equation*}
Letting
\begin{equation*}
\dd\theta = \begin{pmatrix} \dd\theta_0 \\ \dd\theta_1 \\ \vdots \\ \dd\theta_{n} \end{pmatrix}, \, V = \begin{pmatrix} \uu \\ \vv{1} \\ \vdots \\ \vv{n} \end{pmatrix}, \qand \dd V = \begin{pmatrix} \dd \uu \\ \dd \vv{1} \\ \vdots \\ \dd \vv{n} \end{pmatrix},
\end{equation*}
we have
\begin{equation*}
\dd\theta = S' V \dd x + S\dd V = SA V \dd{x} + S \dd{V} = S\left(AV \dd{x} + \dd{V}\right),
\end{equation*}
and hence
\begin{align*}
\dd\theta_0^2 + \dots + \dd\theta_{n}^2 &= (\dd \theta)^T \dd\theta = \left(AV \dd{x} + \dd{V}\right)^T S^T S \left(AV \dd{x} + \dd{V}\right) \nonumber \\ 
&= \left(AV \dd{x} + \dd{V}\right)^T \left(AV \dd{x} + \dd{V}\right).
\end{align*}
With this coordinate change our metric becomes the warped product
\begin{equation}\label{eq:warpedmetric}
g = \bfun(x,u)^2 \dd{x}^2 + \dd{\theta_0}^2 + \dots + \dd{\theta_{n}}^2,
\end{equation}
where \(\uu\) can be obtained from the first entry in \(S^{-1}(x)\theta\).

To see that \( g \) is complete, suppose that \(\gamma\) is a path in \(\R^{n+2}\) with finite length \(L\) but no limit point. 
Then since our metric is described by \eqref{eq:warpedmetric} we know that the \(\theta_i\) coordinates must all be bounded along the path \(\gamma\). 
Since \( S(x) \in SO(n+1) \), the \(\uu\) coordinate must also be bounded\textemdash say \(\abs{\uu} \leq R\) on \(\gamma\) for some \(R \in \R\). 
Now let
\begin{equation*}
M = \min_{\uu \in [-R,R]} r(\uu).
\end{equation*}
Then since \(\bfun(x,\uu) \geq r(\uu)\) it follows that \(\bfun(x,\uu) \geq M > 0\) on \(\gamma\). 
Then
\begin{equation*}
L \geq \int_\gamma \abs{\bfun(x,\uu)} \abs{\dot{x}} \dd{t} \geq \int_\gamma M \abs{\dot{x}} \dd{t} 
\end{equation*}
and hence the \(x\) coordinate must also be bounded as it cannot change more than \(\frac{L}{M}\).
However, if all coordinates are bounded then \(\gamma\) must have a limit in \(\R^{n+2}\). 
Hence our metric is complete. 
\end{proof}

\begin{rem}
	If \((c_1, c_2)\) is not all of \(\R\) then we can show the metric is not complete.
	Without loss of generality take \(a_1 \leq 0 \leq a_2 < \infty\). 
	Consider \(\gamma(x) = (x,0,0, \dots, 0)\) for \(x \in [0, a_2)\). 
	The length of this path is \(\int_0^{a_2} \abs{\bfun(x,0)} \dd{x} = \int_0^{a_2} \dd{x} = a_2 < \infty\) but has no limit in \((c_1, c_2) \times \R^{n+1}\). 
\end{rem}

\begin{rem}
Imposing an upper bound on the scalar curvature of say \(-2\) is not sufficient for completeness. 
An example is \(\bfun(x,u) = e^{-(x^2+1)u}\), which has \(\textrm{Scal} \leq -2\) but is not complete. 
It is critical to bound the geodesic curvature \(\kgam\) along the path \(\uu = \vv{i} = 0\) as well.
\end{rem}

Since Lemma \ref{lem:bfunmetricprops} tells us that the scalar curvature is given by \(-2\frac{\bfun_{\uu \uu}(x,\uu)}{\bfun(x,\uu)}\), it is a function of only \(x\) and \(\uu\) so we may write it as \(\scal(x,\uu)\). 

\begin{prop}\label{prop:bfuncurvbounds}
	If
	\begin{equation*}
		\scal < 0, \quad \abs{\kgam (x)} \leq \sqrt{\frac{1}{2} \abs{\scal(x,\uu)}}, \qand \abs{\kgam (x)} \leq \Lambda
	\end{equation*}
	for all \( x\) and \( \uu \) and for some constant \( \Lambda \), then the metric on \( \R^{n+2} \) is complete.
\end{prop}

\begin{proof}
We will prove this using the equivalent assumption that there is some positive function \( \lambda(x) \) such that 
\begin{equation*}
	\scal(x,\uu) \leq -2 \lambda(x)^2 \quad \text{and} \quad \abs{\kgam} \leq \lambda(x) \leq \Lambda.
\end{equation*} 

Consider the geodesic \(\phi(t) = (x, t, 0, \dots, 0)\). 
Since \(\bfun\) satisfies a Jacobi equation
\begin{equation*}
\bfun_{\uu \uu} + \sec(e_1, e_2) \bfun = 0,
\end{equation*}
for each \(x\),  \(J(t) = \bfun(x, t) e_2(\phi(t))\) is a Jacobi field along \(\phi\) as \( \nabla_{e_1} e_2 = 0 \). 
Note that \(\abs{J(0)} = \abs{\bfun(x,0)} = 1\) and \(\abs{J}'(0) = \bfun_{\uu}(x,0) = \kgam(x)\). 
 
Since \(\sec(e_1, e_2) = \frac{1}{2} \scal(x,\uu) \leq -\lambda(x)^2\), Jacobi field comparison imply that
\begin{align*}
\abs{J(t)} = \abs{\bfun(x,t)} &\geq \abs{J(0)} \cosh(\lambda(x)t) + \frac{\abs{J}'(0)}{\lambda(x)} \sinh(\lambda(x)t) \\
&= \cosh(\lambda(x)t) + \frac{\kgam}{\lambda(x)} \sinh(\lambda(x)t).
\end{align*}
Since \(\abs{\kgam} \leq \lambda(x)\), this is bounded below by \(e^{-\lambda(x) \abs{\uu}}\), and as \(\lambda(x) \leq \Lambda\) we can take \(r(u) = e^{-\Lambda \abs{\uu}}\).
Thus Lemma \ref{lem:bfuncompleteness} shows the metric is complete.
\end{proof}

This gives completeness in terms of an inequality between the geodesic curvature \(\kgam\) and the scalar curvature; however, we still require a uniform bound on the geodesic curvature.
It turns out we can remove the need for this bound if we strengthen the inequality.

\begin{prop}\label{prop:epsiloncomplete}
	If \( \scal < 0 \) and if there is an \( \varepsilon > 0 \) such that
	\begin{equation*}
		\abs{k_\gamma(x)} \leq (1-\varepsilon) \sqrt{\frac{1}{2} \abs{\scal(x,\uu)}},
	\end{equation*}
	then the metric on \(\R^{n+2}\) is complete.
\end{prop}

\begin{proof}
We prove this again under the equivalent assumptions that there exists a positive function \( \lambda(x) \) such that
\begin{equation*}
	\scal(x,\uu) \leq -2\lambda(x)^2 \quad \text{and} \quad \frac{\abs{\kgam}}{\lambda(x)} \leq 1-\varepsilon.
\end{equation*}
As in the proof of Proposition \ref{prop:bfuncurvbounds}, we get
\begin{equation*}
\abs{J(t)} \geq \cosh(\lambda(x)t) + \frac{\kgam}{\lambda(x)} \sinh(\lambda(x)t) =: f(t). 
\end{equation*}
Thus, since \(\frac{\abs{\kgam}}{\lambda(x)} \leq 1-\varepsilon\), the Jacobi field is bounded away from 0. 
In fact, straightforward calculation shows that the \( f(t) \geq \sqrt{1-\varepsilon^2} \) and hence Lemma \ref{lem:bfuncompleteness} implies that the metric is complete.
\end{proof}

One simple case is if we set \(\lambda(x) = 1\), in which we get:
\begin{cor}\label{cor:pathcurvgivescomplete}
If \(\scal \leq -2\) and \(\abs{\kgam} \leq 1\), then the metric is complete.
\end{cor}

\begin{cor}
	If the metric is curvature homogeneous, say \( \scal = -2 \), then \( \bfun(x,\uu) = \cosh(\uu) + \kgam(x) \sinh(\uu) \) and hence the metric is complete if and only if \( \abs{\kgam} \leq 1 \). 
\end{cor}
For the converse we just observe that \( \abs{\kgam} > 1 \) at some point implies that \( \bfun \) must vanish somewhere by the Rauch comparison theorem.


\section{Proof of Theorem \ref{thm:themetriconaV}}\label{sec:onmanifold}
Recall that we are considering a complete simply connected conullity two manifold \(M\) and a connected component \( V \) of \(M_F\) where all \(a_i \neq 0\). 

\begin{proof}
	Pick \( p \in M_F \) and let \( \gamma \colon (c_1, c_2) \to V \) be a maximal integral curve of \( e_2 \) through \( p \) contained in \( V \).
	By the construction in section \ref{sec:prelim} we have well-defined vectors \( e_1\), \(T_i \) and functions \( a_i \) for \( i = 1, \dots, n \).
	Let \( f_i(x) \coloneqq a_i(\gamma(x)) \).
	Now we define coordinates on \( V \) by \( \phi \colon (c_1, c_2) \times \R^{n+1} \to V \) where
	\begin{equation}\label{eq:coordsonV}
		\phi(x, \uu, \vv{1}, \dots, \vv{n}) = \exp_{\gamma(x)}(\uu e_1 + \vv{i} T_i).
	\end{equation}
We claim that \( \phi \) is an isometry from \( N \coloneqq (c_1, c_2) \times \R^{n+1} \) equipped with the metric \eqref{eq:bmetric} onto \( V \). 
Notice that \( \phi \) is a bijection since Proposition \ref{prop:curvehitsonce} guarantees that each leaf of \( \mathcal{F} \) intersects \( \gamma \) exactly once and \( \{e_1, T_i\} \) span the target space of \( \mathcal{F} \).
It remains to show that \( \phi \) is a local isometry. 

The existence of solutions to ODEs gives \( \bfun \colon (c_1, c_2) \times \R \to \R \) as the function which satisfies \( \bfun(x,0) = 1 \), \( \bfun_{\uu}(x,0) = -k_{\gamma}(x) \), and \( \bfun_{\uu \uu}(x,\uu) = -\frac{1}{2}\bfun(x,\uu)\scal(x,\uu)  \), where \( k_\gamma \) is the geodesic curvature of the curve \( \gamma \) and \( \scal(x,\uu) \) is the scalar curvature of \( V \) at \( \exp_{\gamma(x)}(\uu e_1) \).

The coordinates \ref{eq:coordsonV} define vector fields \(\pdv{x}\), \(\pdv{\uu}\), and \(\pdv{\vv{i}}\) on \(V\).
Computations with Jacobi fields show that \( \pdv{\uu} \) and \( \pdv{\vv{i}} \) at \( \exp_{\gamma(x)}(ue_1 + v_i T_i) \) are the parallel translations of \( e_1 \) and \( T_i \) along the geodesic \(  \exp_{\gamma(x)}(t(ue_1 + v_i T_i)) \). 
By propositions \ref{prop:initialcurvcomps} and \ref{prop:niceremainingT}, \( \pdv{\uu} = e_1 \) and \( \pdv{\vv{i}} = T_i \). 

Now that we've verified \( \pdv{\uu} = e_1 \) and \( \pdv{\vv{i}} = T_i \) we use Jacobi fields again to compute \(\pdv{\phi}{x}\). 
Fix \(p = (x, {\uu}, {\vv{i}})\) and consider the family of geodesics
\begin{equation}
	\alpha_s(t) = \phi(x + s, t {\uu}, t {\vv{i}}).
\end{equation}
Let \(J(t)\) be the Jacobi field along \(\alpha_0(t)\) corresponding to the variation \(\alpha_s\). 
Then we have
\begin{equation*}
	J(0) = \gamma'(0) = e_2
\end{equation*}
and also
\begin{align}\label{eq:onmfdJacobiprime}
	J'(0) &= \frac{D}{\dd s} \pdv{\alpha_s}{t} \eval_{s=0, t=0} = \nabla_{e_2}({\uu} e_1 + {\vv{i}} T_i) \\
	&= -{\uu} \beta e_2 - a_1 {\vv{1}} e_1 + ({\uu} a_1 - {\vv{2}} a_2) T_1 \nonumber \\
	& \quad + \sum_{i=2}^{n-1} ({\vv{i-1}}a_i - {\vv{i+1}}a_{i+1})T_i +  {\vv{n-1}} a_{n} T_{n}. \nonumber
\end{align}
From the Jacobi equation, the only nonzero component of \(J''\) is the \(e_2\) component since all other curvature terms become zero. 
Note that \(\dot{\alpha}_0(t) = {\uu} \pdv{\uu} + {\vv{i}} \pdv{\vv{i}}\) but only the first term survives after substituting into the curvature tensor.
Let \(f(t)\) be the \(e_2\) component of \(J(t)\).
Then we have
\begin{equation*}
	f''(t) = \expval{J''(t), e_2} = -\expval{R\left(J(t), {\uu} \pdv{\uu}\right) {\uu} \pdv{\uu}, e_2} = -u^2 \sec(e_1, e_2) f(t).
\end{equation*}
Observe that \( f(t) = \bfun(x, t\uu) \) satisfies the Jacobi equation.
Keeping in mind that the \( a_i \) in \ref{eq:onmfdJacobiprime} are evaluated at \( \gamma(x) \) and \( a_i(\gamma(x)) = f_i(x) \), the Jacobi field is given by
\begin{equation*}
	J(t) = \bfun(x, t\uu) e_2 - t \vv{1} f_1(x) e_1 
	+ \sum_{i=1}^{n} t({\vv{i-1}}f_i(x) - {\vv{i+1}}f_{i+1}(x))T_i. 
\end{equation*}
Thus,
\begin{equation*}
	\pdv{\phi}{x}\eval_p = J(1) = \bfun(x,\uu) e_2 - {\vv{1}} f_1(x)e_1
	+ \sum_{i=1}^{n} ({\vv{i-1}} f_i(x) - {\vv{i+1}} f_{i+1}(x))T_i,  
\end{equation*}
and hence
\begin{align*}
	\expval{\pdv{\phi}{x}, \pdv{\phi}{x}}_p &= (\bfun(x,\uu))^2 
	+ \sum_{i=0}^{n} ({\vv{i-1}} f_i(x) - {\vv{i+1}} f_{i+1}(x))^2  \\ 
	\expval{\pdv{\phi}{x}, \pdv{\phi}{\uu}}_p &= -{\vv{1}} f_1(x). \\
	\expval{\pdv{\phi}{x}, \pdv{\phi}{\vv{i}}}_p &= {\vv{i-1}} f_i(x) - {\vv{i+1}} f_{i+1}(x). 
\end{align*}
Altogether, we have shown that \( \phi \) is a local isometry, and hence an isometry. 
One should note that we must have \( \bfun > 0 \) since \( \bfun \) cannot vanish.
If it did so along some path \( t \mapsto \bfun(x,t) \) then \( \bfun \) would have to change sign.
This would imply that nearby leaves of the foliation \( \leaff \) would cross each other which cannot happen. 
\end{proof}


\section{Foliations of Surfaces and the Smoothness of our Metric} \label{sec:foliations}
In order to construct complete conullity 2 manifolds and prove Theorem \ref{thm:smoothnessresultinintro} we start by investigating foliations of surfaces by geodesics. 

In \cite{Brooks:2021}, Brooks showed that if \(H \colon \R \to \R\) is a locally Lipschitz function and \(\Sigma\) is a complete surface, then for any \((p_0, v_0) \in T_1\Sigma\) there exists a unique arc-length parametrized \(C^{1,1}\) curve \(\gamma \colon \R \to \Sigma\) with \( \gamma(0) = p_0 \), \( \gamma'(0) = v_0 \) whose turning angle at \(\gamma(x)\) is \(H(x)\). 
Furthermore, if \( \Sigma = \mathbb{H}^2 \) the geodesics orthogonal to \( \gamma \) foliate \( \mathbb{H}^2 \) if and only if it has Lipschitz constant 1. 
We now generalize this latter result to a complete simply connected surface \(\Sigma\) with negative curvature \(\scal \leq -2\).
For a curve \(\gamma\), let \(X\) be a unit vector field along \(\gamma\) orthogonal to \(\dot{\gamma}\) and define the family of geodesics
\begin{equation}\label{eq:sigmacoords}
	\sigma_x(\uu) = \sigma(x,\uu) \coloneqq \exp_{\gamma(x)}^\perp(\uu X).
\end{equation}
Fixing \(x\) and varying \(\uu\) traces out geodesics orthogonal to \(\gamma\). 

\begin{prop}\label{prop:curvesfoliate}
	Let \(\Sigma\) be a complete simply connected surface with \( \scal \leq -2 \).
	Fix a point \(p \in \Sigma\) and a unit vector \(v \in T_p \Sigma\). 
	There exists an injective mapping from Lipschitz functions \(H: \R \to \R\) with Lipschitz constant 1 to arc-length parametrized curves \(\gamma\) such that:
	\begin{enumerate}[label = (\alph*)]
		\item
		\(\gamma(0) = p\) and \(\gamma'(0) = v\).
		\item
		The geodesics \(\sigma_x \colon \uu \mapsto \exp_{\gamma(x)}^\perp(\uu X)\) form a foliation of \(\Sigma\). 
	\end{enumerate}
	Here \(H\) is the turning angle of \(\gamma\) and \( H' \) is the geodesic curvature \( \kgam \), defined almost everywhere.
	Furthermore, the metric in the \( (x,\uu) \) coordinates defined by \ref{eq:sigmacoords} is of the form \(\bfun(x,\uu)^2 \dd{x}^2 + \dd{\uu}^2\).
\end{prop}

\begin{proof}
	In \cite{Brooks:2021} the existence of a curve \( \gamma \) with turning angle \( H \) was proved. 
	The proof that the geodesics normal to \( \gamma \) foliate \( \Sigma \) was carried out in the case of \( \Sigma = \mathbb{H}^2 \) and we indicate how to generalize it to the case of \( \scal \leq -2 \).
	For this it is sufficient to show that we get a foliation if \( H \) has Lipschitz constant 1. 
	
	To show we have a foliation we show that for every \(p \in \Sigma\) there is a unique closest point \(q\) on \(\gamma\).
	Let \(\delta(q) \coloneqq d(p,q)\). 
	Then in a normal neighborhood centered around \(p\), which can be taken to be the entire surface \( \Sigma \), \(\delta = r\), so \(\text{grad } \delta = \text{grad } r = \partial_r = (\nabla r)^\#\).
	Then we may relate this to the Hessian by \(\hess_r(w) = \nabla_w \partial_r = \nabla_w (\text{grad } \delta)\).
	We also have that \(\nabla_{\dot{\gamma}} \dot{\gamma} = h(t) (\dot{\gamma})^\perp\) where \((\dot{\gamma})^\perp\) is a unit length vector orthogonal to \(\dot{\gamma}\). 
	
	Let \(L(q) = \cosh(\delta(q)) = \cosh(r)\). 
	Then \( (L \circ \gamma)' = \sinh(r) \expval{\partial_r, \dot{\gamma}} \) and 
	\begin{equation*}
		(L \circ \gamma)'' = \cosh(r)\expval{\partial_r, \dot{\gamma}}^2 + \sinh(r) \expval{\hess_r(\dot{\gamma}), \dot{\gamma}} + \sinh(r)\expval{\partial_r, h(t) (\dot{\gamma})^\perp}.
	\end{equation*}
	
	Since \(\sec \leq -1\),  on \(\Sigma \setminus \{p\}\) we have \(\hess_r \geq \frac{\cosh(r)}{\sinh(r)} \pi_r\) where \(\pi_r\) is the projection to the tangent space of level sets of \(r\). 
	Now \(\expval{\pi_r(\dot{\gamma}), \dot{\gamma}} = \expval{\partial_r^\perp, \dot{\gamma}}^2\), so 
	\begin{align*}
		(L \circ \gamma)'' &\geq \cosh(r)\expval{\partial_r, \dot{\gamma}}^2 + \cosh(r) \expval{\partial_r^\perp, \dot{\gamma}}^2 + h(t) \sinh(r) \expval{\partial_r, (\dot{\gamma})^\perp} \\
		&= \cosh(r) + h(t) \sinh(r) \expval{\partial_r, (\dot{\gamma})^\perp} \geq e^{-r} > 0
	\end{align*}
	where the last bound uses the fact that \(\abs{h(t)} \leq 1\) since \(H\) has Lipschitz constant 1 and \(H'(t) = h(t)\). 
	Hence \((L \circ \gamma)\) has a unique minimum, and since \(\cosh\) is a monotone function, \(\delta\) also has a unique minimum.
	
	To verify the claim about the metric, it suffices to show that \(\pdv{x}\) and \(\pdv{\uu}\) are everywhere orthogonal. 
	Fix \(x\) and consider the Jacobi field \(J(t)\) along \(\sigma_x\) corresponding to the variation of geodesics \(\sigma\). 
	Then \(J(0) = \dot{\gamma}(x) = \pdv{x}\eval_{(x,0)}\) and
	\begin{equation}
			\dot{J}(0) = \frac{D}{\dd{x}} \pdv{\sigma_x(t)}{t} \eval_{t=0} = \nabla_x X = -h(t) \dot{\gamma}(x) = -h(t) \pdv{x}\eval_{(x,0)}.
	\end{equation}
	By the Jacobi equation,
	\begin{equation}
			0 = \expval{\ddot{J}, \pdv{\uu}} + R\left(J, \dot{\sigma}_x, \dot{\sigma}_x, \pdv{\uu}\right) = \expval{\ddot{J}, \pdv{\uu}} + 0
	\end{equation}
	which combined with the initial conditions tells us that \(\expval{J, \pdv{\uu}} = 0\) everywhere.
	As \(\pdv{x}\eval_{(x,\uu)} = J(\uu)\) we have the desired result.
\end{proof}

\begin{rem}
	Given a curve \( \gamma \) whose orthogonal geodesics \( \sigma_x \) form a foliation of \( \Sigma \), its turning angle will be Lipschitz but possibly without Lipschitz constant 1.  
	In the constant curvature case, where \( \Sigma = \mathbb{H}^2 \), this is a bijection.
\end{rem}

\begin{proof}[Proof of Theorem \ref{thm:smoothnessresultinintro}]
	The proof is a straightforward generalization of the proof in \cite{Brooks:2021} and we only indicate the necessary changes. 
	
	Let \(M = \Sigma \times \R^n\) and consider the smooth product metric
	\begin{equation*}
		g_{\Sigma \times \R^n} = g_\Sigma + \dd{\vv{1}}^2 + \dots + \dd{\vv{n}}^2.
	\end{equation*}
	We will write the metric of the form \eqref{eq:bmetric} as a sum of smooth things to show that it is in fact smooth on \(M\). 
	
	Recall that \(S \subset \R\) is the open set of \(x\)-values where \(\gamma(x)\) is smooth.
	This gives two subsets: \(S_\Sigma \subset \Sigma\), the set of points whose \(x\) coordinate lies in \(S\), and \(S_M \subset M\), the set of points whose projection to \(\Sigma\) lie in \(S_\Sigma\). 
	We define smooth coordinates on each connected component of \(S_M\) by extending the \( (x,\uu) \) coordinates on \( \Sigma \) by picking an orthonormal frame \( \partial{\vv{i}} \) of \( \R^n \) along \( \gamma \) to get coordinates \( (x,\uu,\vv{i}) \).
	
	The following computations are performed in the product metric where \( a_i = 0 \) so \( \nabla_{e_2}e_2 = \beta e_1 \) and \( \nabla_{e_2}e_1 = -\beta e_2 \) are the only nonzero covariant derivatives. 
	Furthermore, here \( e_2 = \frac{1}{\bfun} \pdv{x} \).
	We now want to modify the product metric.
	On points in \(S_M\), define
	\begin{align*}
		g_{f_i} = -f_i(x) \vv{i} (\dd{x} \dd{\vv{i-1}} + \dd{\vv{i-1}}\dd{x}) &+ f_i(x) \vv{i-1}(\dd{x}\dd{\vv{i}} + \dd{\vv{i}} \dd{x}) \\
		&+ f_i(x)^2 (\vv{i-1}^2 + \vv{i}^2) \dd{x}^2 \nonumber
	\end{align*}
	and
	\begin{equation*}
		\mixedfactor{i} = -2 \vv{i-1}\vv{i+1} f_i(x) f_{i+1}(x) \dd{x}^2,
	\end{equation*}
	while on the complement of \(S_M\) let both be zero.
	Then
	\begin{equation*}
		g = g_{\Sigma \times \R^n} + \sum_{i=1}^n g_{f_i} + \sum_{i=1}^{n-1} \mixedfactor{i}
	\end{equation*}
	is the metric in \eqref{eq:bmetric} on \(S_M\), and it suffices to show that \(g_{f_i}\) and \(\mixedfactor{i}\) are smooth. 
	
	Fix smooth vector fields \(X_1\), \(X_2\), and let
	\begin{equation*}
		F_i \coloneqq g_{f_i}(X_1, X_2) \quad \text{and} \quad \mixedfactorF{i} \coloneqq \mixedfactor{i}(X_1, X_2).
	\end{equation*}
	In other words we have
	\begin{align*}
		F_i = &-\frac{f_i(x)}{\bfun(x,\uu)} \vv{i} (\expval{X_1, e_2}\expval{X_2, \partial_{\vv{i-1}}} + \expval{X_2, e_2}\expval{X_1, \partial_{\vv{i-1}}}) \\
		&+ \frac{f_i(x)}{\bfun(x,\uu)} \vv{i-1} (\expval{X_1, e_2}\expval{X_2, \partial_{\vv{i}}} + \expval{X_2, e_2}\expval{X_1, \partial_{\vv{i}}}) \nonumber \\ 
		&+ \frac{f_i(x)^2}{\bfun(x,\uu)^2}(\vv{i-1}^2 + \vv{i}^2) \expval{X_1, e_2}\expval{X_2, e_2} \nonumber
	\end{align*}
	and
	\begin{equation*}
		\mixedfactorF{i} = -2 \vv{i-1} \vv{i+1} \frac{f_i(x) f_{i+1}(x)}{\bfun(x,\uu)^2} \expval{X_1, e_2}\expval{X_2, e_2}.
	\end{equation*}
		Jacobi field estimates imply that since \( \scal \leq -2 \) and \( \bfun(x,0) = 1 \) we have \( \bfun(x,\uu) \geq \cosh \uu + \kgam \sinh \uu \) and hence \( \bfun \) is bounded away from 0 on any set where \( \uu \) is bounded.
		It is then straightforward to check that these functions and their derivatives satisfy the following properties on \( S_M \):
	\begin{enumerate}[label = (\alph*)]
		\item
		They are rational functions in \(\uu, \vv{i}\), derivatives of \(f_i(x)\) and \(\bfun(x,\uu)\), and inner products \(\expval{X, e_j}\) or \(\expval{X, \partial_{\vv{i}}}\). 
		\item
		The denominator is bounded away from zero on any set where \(\abs{u}\) is bounded.
		\item
		Each term of each numerator has some positive power of some \(f_i^{(j)}(x)\).
	\end{enumerate}
	These properties, along with the condition \eqref{eq:smoothnesscondition} in the statement of Theorem \ref{thm:smoothnessresultinintro}, are enough to show that these functions smoothly extend to \(M\) by defining them to be zero on \(M \setminus S_M\). 
	Properties (a) and (c) show that the numerators of the functions go to zero in the limit while (b) shows that the denominators are bounded away from zero.
	Therefore \(g\) is smooth.
	Furthermore, the extension to \(M \setminus S_M\) has \( \scal \leq -2 \) by continuity, so we have a smooth conullity 2 metric on \(M\). 
	
	To verify completeness we use a similar argument as that in the proof of Lemma \ref{lem:bfuncompleteness}.
	The coordinate change is identical, but in lieu of a lower bound on \( \bfun \) we use completeness of the surface \( \Sigma \) with metric \( g_\Sigma = \bfun^2 \dd{x}^2 + \dd{\uu}^2 \). 
	Given a path \( \gamma \) with finite length we still have that the \( \theta_i \) coordinates, and thus the \( \uu \) coordinate, are bounded.
	To bound the \( x \) coordinate, project to a corresponding curve \( \gamma_\Sigma \) in \( \Sigma \) via \( (x,\uu,\vv{i}) \mapsto (x,\uu) \).
	This preserves finiteness of the length.
	Then by completeness of \( \Sigma \) the \( x \) coordinate must also be bounded. 
\end{proof}

\begin{rem}
	Notice that in Theorem \ref{thm:smoothnessresultinintro}, the set \( S_M \) may be larger than the set \( M_F \) where the Frenet frame is defined since we allow the \( f_i(x) \) to vanish (even on open subsets).
	In particular we can glue pieces which are locally irreducible with some that are locally product metrics \( \Sigma \times \R^n \).
\end{rem} 

While the nullity vectors \( T_2 \), \( \dots \), \( T_n \) are not necessarily defined outside of \( M_F \), we may ask if they have one-sided limits as we approach the boundary of \( M_F \). 
They in fact do, supposing that the \( f_i \) are bounded in the region as is the case in Theorem \ref{thm:smoothnessresultinintro}.
\begin{lem}\label{lem:onesidedlimits}
	Let \(V\) be a connected component of \(M_F\).
	For \(p \in V\), construct coordinates on \(V\) as above, using \(\gamma\), an integral curve of \(e_2\) passing through \(p\).
	Let \(x_* \in \R\) be such that \(\gamma(x_*) = p_* \notin M_F\), a boundary point of \(V\).
	If the functions \(f_i\) are bounded near \(p_*\) then \(T_1, \, \dots, \, T_n\) have one-sided limits along \(\gamma\) as \(x \to x_*\).
\end{lem}

\begin{proof}
	As all \(f_i\) are bounded near \(p_*\) it follows that the derivatives of all \(T_i\) are bounded near \(p_*\), so say they are bounded in absolute value by a constant \(M\).
	As the \(T_i\) are unit vectors they live in a compact subset so they admit convergent subsequences, hence it suffices to show that the limit is unique. 
	
	Suppose we have two sequences of real numbers \(x_n\) and \(y_n\) such that both converge to \(x_*\) and \((T_i)_{x_n} \to X\), \((T_i)_{y_n} \to Y\).
	Taking a coordinate patch of \( M \) centered at \( p_* \) allows us to treat \( T_i \) as taking values in \( \R^n \) and therefore we may apply the mean value inequality
	\begin{equation*}
		\abs{f(b)-f(a)} \leq (b-a) \sup\abs{f'}
	\end{equation*}
	to the functions \(f = T_i\) with \(a = x_n\) and \(b = y_n\).
	As we take \(n \to \infty\), the right side converges to zero since \(b-a \to 0\) while \(\sup\abs{f'}\) is some finite value, so the left side must also converge to zero so \(X = Y\). 
\end{proof}

\begin{rem}
	As the modifications to the product metric smoothly go to zero as we approach \(M \setminus S_M\) we may have different modifications on each connected component of \(S_M\). 
	In particular, we may choose different orthonormal frames \(\partial_{\vv{1}}, \dots, \partial_{\vv{n}}\)
	which do not extend continuously. 
	For example, if two components are adjacent, the frames have a limit at the common boundary by Lemma \ref{lem:onesidedlimits}, but these limits may differ by an element of \( O(n) \).
Similarly, we may, under certain conditions, change the basis within a component of \( S_M \) if one \( f_i \), but not necessarily all, vanish.
\end{rem}


\section{Fundamental Group}\label{sec:fungroup}
We now investigate topological properties of our conullity two manifolds. 
This is a partial generalization of Brooks' results \cite{Brooks:2021}, focusing on locally irreducible manifolds.
Let \( M \) be a complete, locally irreducible \( (n+2) \)-manifold with conullity 2 and \( \scal < 0 \) such that the scalar curvature is constant along nullity geodesics.
Its universal cover \( \tilde{M} \) is diffeomorphic to \( \R^n \) since \( \sec \leq 0 \) and we denote by \( G \) the deck group of the universal cover acting properly discontinuously and isometrically on \( \tilde{M} \).
Recall that \( \tilde{M}_F \subset \tilde{M} \) is the subset on which all \( a_i \neq 0 \).

\begin{lem}\label{lem:Adef}
For \(p, q \in \tilde{M}\), let
\begin{equation*}
A(p,q) = \int \abs{a_1(\gamma(t)) \expval{\gamma', e_2}} \dd{t},
\end{equation*}
integrating over the unique geodesic segment \(\gamma\) from \(p\) to \(q\). 
Then \(A(p,q)\) is an isometry invariant depending only upon the leaves \(\mathcal{F}_p\) and \( \mathcal{F}_q\).
Furthermore, if \( M \) is locally irreducible then \( p, q \in \tilde{M} \) are in the same leaf if and only if \( A(p,q) = 0 \).
\end{lem}

\begin{proof}
Since \(e_2\) is well-defined whenever \(a_1 \neq 0\) and \(\abs{\expval{\gamma', e_2}} \leq \| \gamma' \|\) this integral is well-defined.
Furthermore, \(A(p,q) = A(gp, gq)\) for any \(g \in G\) since \( \abs{a_1} \) is an isometry invariant.

To see that it only depends on the leaves, let \(p, \, p_2 \in \mathcal{F}_p\),  \(q, \, q_2 \in \mathcal{F}_q\), \( \gamma \) be the geodesic between \( p \) and \( q \), and \( \gamma_2 \) be the geodesic between \( p_2 \) and \( q_2 \).
The leaves through \(\gamma\) are the leaves between \(\mathcal{F}_p\) and \(\mathcal{F}_q\), which is precisely the same set of leaves as those through \(\gamma_2\). 
We split up the integrals along \(\gamma\) and \(\gamma_2\) into subintervals where all \(a_i \neq 0\) and these two collections of subintervals are in bijection with one another. 
With this in mind, it suffices to verify that the assertion is true along an interval where \(a_i \neq 0\).
We show this in the case of the interval \( [0,1] \) with our path increasing in the \( x \) coordinate.

On each connected component of \( \tilde{M}_F \) we have coordinates \((x, \uu, \vv{1}, \dots, \vv{n})\) as before where the metric has the form \eqref{eq:bmetric}.
 By Theorem \ref{lem:bfunmetricprops}, \(a_1 = \frac{f_1(x)}{\bfun(x,u)}\) and using the expression for \( e_2 \) we get
\begin{equation*}
\gamma'(t) = \dd{x}(\gamma') \pdv{x} + \dd{\uu}(\gamma') \pdv{\uu} + \sum_{i=1}^n \dd{\vv{i}}(\gamma') \pdv{\vv{i}}
\end{equation*}
where \(\dd{x}(\gamma') \geq 0\) everywhere.
Only the \(x\)-component of this does not cancel out when we compute \(\expval{\gamma', e_2}\), and
\begin{align*}
\int_0^1 \abs{a_1 \expval{\gamma'(t), e_2}} \dd{t} &= \int_0^1 \abs{\frac{f_1(x)}{\bfun(x,\uu)} \frac{\dd{x}(\gamma')}{\bfun(x,\uu)} \bfun(x,\uu)^2} \dd{t} \\
&= \int_0^1 \abs{f_1(x(\gamma(t)))} \dd{x}(\gamma') \dd{t} \\
&= F_1(x(\gamma(1))) - F_1(x(\gamma(0)))
\end{align*}
where \(F_1\) is an antiderivative of \(\abs{f_1}\). 
Hence the value of the integral only depends on the \(x\)-coordinates, and therefore the leaves, of the endpoints. 

The above argument did not require \( \gamma \) to be a geodesic, only that it was monotone in the \( x \) coordinate and so did not backtrack into previous leaves. 
Then by Proposition \ref{prop:curvehitsonce}, integral curves of \( e_2 \) satisfy this as well. 
If \( p, q \) are in the same leaf then their minimal geodesic is contained in the same leaf and hence the integral is zero. 
If not, then let \( \gamma \) be an integral curve of \( e_2 \) through \( p \) which must intersect the leaf \( \leaff_q \) at some other point. 
Then the computation of \( A(p,q) \) amounts to integrating \( \abs{f_1(x)} \), which is positive almost everywhere, over an interval of nonzero measure, and hence this integral is nonzero.
\end{proof}

\begin{rem}
	Since the argument didn't require \( \gamma \) to be a geodesic, only that it did not backtrack through leaves, \(A(p,q)\) may be computed by any path between \(p\) and \(q\) with this property.
\end{rem}

Now we examine how \( G \) acts on the leaves \(\mathcal{F}_p\).
As \( \sec \leq 0 \), \(G\) cannot have torsion.
We will show that no nontrivial \(g \in G\) can fix a leaf, so \(g(\mathcal{F}_p) \subset \mathcal{F}_p\) implies \( g = e \). 
One tool we will use is knowledge about crystallographic groups. 
A crystallographic group on \(\mb{R}^n\) is a discrete uniform subgroup of the Euclidean group \(E(n)\), which is equivalent to the assumption that the group acts properly discontinuously with compact quotient on \(\R^n\). 
Moreover, crystallographic groups have a finite index normal subgroup which acts by translations, see \cite{Wolf:2010}.
We now use this to prove that non-identity elements cannot fix two leaves:

\begin{lem} \label{lem:fix2leaf}
If \(g \in G\) fixes two distinct leaves \(\leaf{0}\), \(\leaf{1} \in \mathcal{F}\) then \(g = e\).
\end{lem}

\begin{proof}
Suppose otherwise.
Then \(g(\leaf{0}) \subset \leaf{0}\) and \(g(\leaf{1}) \subset \leaf{1}\), and the restrictions \(g_{\leaf{i}}\) are acting by isometries of \(\R^n\).
Consider the coverings \(\pi_i \colon \leaf{i} \to \leaf{i}/\langle g \rangle\), where \( \langle g \rangle \) is the subgroup of \( G \) generated by \( g \).
As the leaves are complete flat totally geodesic hyperplanes, the orbit spaces \(\leaf{i}/\langle g \rangle\) are complete flat totally geodesic submanifolds of \(\tilde{M}/\langle g \rangle\).
Then we may apply the Soul theorem to the orbit spaces, giving us closed totally convex, totally geodesic embedded submanifolds \(S_i \subset \leaf{i}/\langle g \rangle\). 
We may then consider the preimages of the souls under the projection map: \(N_i \coloneqq \pi_i^{-1}(S_i)\), which are totally convex totally geodesic submanifolds contained in the leaves \(\leaf{i}\).
Note that these \(N_i\) are connected as they are totally convex.
Furthermore, since \( \langle g \rangle \simeq \Z \), each \( S_i \) is a closed geodesic and each \( N_i \) is a geodesic isometric to \( \R \) so we will also call them \( \gamma_0 \) and \( \gamma_1 \).

Now \(\langle g \rangle\) acts properly discontinuously and with compact quotient on each \(\gamma_i\). 
Hence it acts by translations on these two geodesics. 

Since \(\sec \leq 0\), \(\tilde{M}\) is convex.
Also, \(\gamma_1\) is a closed convex subset, so \(d(\cdot, \gamma_1)\) is a convex function on \(\tilde{M}\). 
Now pick \(p_0 \in \gamma_0\). 
Then \(g\) acts on \(p_0\) by translation along \(\gamma_0\). 
As \(g\) translates along \(\gamma_0\), \(g\) fixes the image of \(\gamma_0\).
Furthermore, \(d(\cdot, \gamma_1)\) is constant on \(\gamma_0\) since \(d(p_0, \gamma_1) = d(g^i(p_0), \gamma_1)\) for any \(i\) since \(g\) preserves \(\gamma_1\), acts by isometries, and \( d \) is convex.

Let \(p_1 \in \gamma_1\) be the closest point to \(p_0\).
Then \(p_1\) is translated along \( \gamma_1 \) by \(g\). 
Then for any \(i\), \(d(g^i(p_0), g^i(p_1)) = d(p_0, p_1)\) and hence \(\gamma_0\) and \(\gamma_1\) are parallel geodesics. 
By results from section 2 of \cite{Ballman:1985}, the geodesics \( \gamma_0 \) and \( \gamma_1 \) bound a flat totally geodesic strip \(V\). 
Since \(V\) is flat, at every point \(q \in V\) \(T_qV\) contains a (nonzero) nullity vector \(T\) which must be some combination \(\alpha_1 T_1 + \dots + \alpha_n T_n\). 
Furthermore the tangent space \(T_qV\) must contain a vector \(v = b_1 e_1 + e_2 + c_1 T_1 + \dots + c_n T_n\) with a nonzero \(e_2\) component since \(\gamma_0\) and \(\gamma_1\) are in different leaves. 
Then we may consider \(\nabla_v T \in T_qV\). 
Recall that \(\nabla_{e_1} T_i = 0\) and \(\nabla_{T_i} T_j = 0\) for any \(i, j\). 
Therefore we get
\begin{align*}
\nabla_v T &= (\alpha_1 \nabla_{e_2} T_1 + \dots + \alpha_n \nabla_{e_2} T_n) + (v(\alpha_1) T_1 + \dots + v(\alpha_n) T_n)\\
&= -a_1 \alpha_1 e_1 + (-a_2 \alpha_2 + v(\alpha_1)) T_1 \\
&\quad + \sum_{i=2}^{n-1} (\alpha_{i-1} a_i - \alpha_{i+1}a_{i+1} + v(\alpha_i)) T_i + (\alpha_{n-1}a_n + v(\alpha_n))T_n.
\end{align*}
Since \(M\) is locally irreducible \( \tilde{M}_F \subset \tilde{M} \) is an open dense set with all \(a_i \neq 0\).
Suppose \(q \in \tilde{M}_F \cap V\), which we may do since \( V \) spans an interval of leaves, a dense set of which lie in \( \tilde{M}_F \).
We will prove inductively that each \( \alpha_i = 0 \) and \( v(\alpha_i) = 0 \), in which case \( T = 0 \) which contradicts our assumption.
Since \( V \) is totally geodesic, \(\nabla_v T \in \text{span}\{v, T\}\). 
However, \(v\) has a nonzero \(e_2\) component while \(T\) and \(\nabla_v T\) do not, so \(\nabla_v T \in \text{span}\{T\}\). 
Then the \(e_1\) component of \(\nabla_v T\) is zero, so \(-a_1 \alpha_1 = 0\). 
Since all \(a_i \neq 0\), we must have \(\alpha_1 = 0\).
This is true at every point in \( \tilde{M}_F \), in particular in a neighborhood around \( q \), so \(v(\alpha_1) = 0\) as well.
This proves the base case.

For the induction step, if we suppose that \( \alpha_1, \, \dots, \, \alpha_i\), \(v(\alpha_1), \, \dots, \, v(\alpha_i) = 0 \) then the \( T_i \) component of \( T \) is zero and hence the \( T_i \) component of \( \nabla_v T \) must also be zero.
We therefore have that \( -\alpha_{i+1} a_{i+1} = 0 \) and so \( \alpha_{i+1} = 0 \), and since this is true in an open set we also have that \( v(\alpha_{i+1}) = 0 \), completing the induction argument. 
\end{proof}

\begin{lem}\label{lem:fix1leaf}
If \(g \in G\) fixes a leaf \(\leaf{0}\) then \(g = e\).
\end{lem}

\begin{proof}
We will show that \(g\) must also fix a nearby leaf and then use Lemma \ref{lem:fix2leaf}.
Let \( p_0 \in \leaf{0} \) and let \( \gamma \) be an integral curve of \( e_2 \) starting at \( p_0 \).
Then \( A(\leaf{0}, \leaff) \) is the integral of \( \abs{a_1} \) along \( \gamma \) from \( p_0 \) to the leaf \( \leaff \).
Note that \( \gamma \) intersects every leaf by Proposition \ref{prop:curvehitsonce}.
Let \( \tilde{\leaff} \) be a nearby leaf. 
Suppose that \( \tilde{\leaff} \) and \( g(\tilde{\leaff}) \) intersect \( \gamma \) at \( \gamma(s) \) and \( \gamma(t) \) respectively. 
Without loss of generality suppose \( s < t \). 
Then \( A(\leaf{0}, \tilde{\leaff}) =  A(\leaf{0}, g(\tilde{\leaff}))\), so the integrals are the same. 
This implies that \( s = t \) since \( \text{Im}(\gamma) \cap M_F \) is dense in \( \text{Im}{\gamma} \) so the integral over any nonempty open interval is nonzero.
Hence \( \tilde{\leaff} \) is fixed by \( g \).
\end{proof}

We are now ready to prove Theorem \ref{thm:fundgroup}. 
\begin{proof}[Proof of Theorem \ref{thm:fundgroup}]
Pick a leaf \(\leaf{0}\) of the foliation \( \mathcal{F} \) and define a real-valued function on the set of leaves
\begin{equation*}
B(\leaff_p) \coloneqq \pm A(\leaf{0}, \leaff_p)
\end{equation*}
where we assign a positive value on one side of \(\leaf{0}\) and a negative value on the other.
By an argument similar to Lemma \ref{lem:Adef}, \( B \) is an injection from the set of leaves to \( \R \).
Similarly we may define
\begin{equation*}
A^*(p,q) \coloneqq \pm A(p,q)
\end{equation*}
where we assign a positive or negative sign depending on the relative positions of \(\leaff_p\) and \(\leaff_q\) and such that it agrees with the choice relative to \(\leaf{0}\).
That is, \(B(\leaff_p) = A^*(p_0, p)\) where \(p_0 \in \leaf{0}\). 

Now recall that the value of \(A(p,q)\) is path-independent as long as the path does not cross back over previously traversed leaves.
So if we take points \(p_1\), \(q_0\), \(q_1 \in \tilde{M}\) with \(p_1\), \(q_0\), \(q_1\) all on the same side of \(\leaf{0}\), then
\begin{align*}
A^*(p_0, q_0) - A^*(p_0, q_1) &= (A^*(p_0, p_1) + A^*(p_1, q_0)) - (A^*(p_0, p_1) + A^*(p_1, q_1)) \nonumber \\
&= A^*(p_1, q_0) - A^*(p_1, q_1).
\end{align*}
We may remove the assumption that all points are on the same side of \(p_0\) by applying this multiple times. 
It follows that
\begin{align}
B(g(\leaff_p)) - B(g(\leaff_q)) &= A^*(\leaf{0}, g(\leaff_p)) - A^*(\leaf{0}, g(\leaff_q)) \\
&= A^*(g(\leaf{0}), g(\leaff_p)) - A^*(g(\leaf{0}), g(\leaff_q)) \nonumber \\
&= A^*(\leaf{0}, \leaff_p) - A^*(\leaf{0}, \leaff_q) \nonumber\\
&= B(\leaff_p) - B(\leaff_q). \nonumber
\end{align}
We therefore get an induced action of \( G \) on the image of \( B \) given by \( g \cdot B(\leaff_p) := B(g(\leaff_p)) \).
Using this definition, \( G \) acts fixed-point freely since \( g \in G \) cannot fix a leaf without being the identity element, and the action is by isometries in the absolute value metric on \( \R \).
So if \(G \neq \{e\}\) then the image of \(B\) is all of \(\R\) since \( G \) acts on it by nontrivial translations and by the definition of \( B \) its image is connected.

If \(G\) acts properly discontinuously on \(\R\) then it must be trivial or \(\Z\), so suppose it does not.
Such an action on \(\R\) must then have every orbit dense. 
Recall that \(a_1\) is a continuous function, so if it ever vanishes on some leaf then it must vanish on a dense set and then be identically zero.
But this contradicts the assumption that \(M\) is locally irreducible, so without loss of generality \(a_1 > 0\) everywhere.
Since \(a_1\) never vanishes, \(T_1\) is defined everywhere and so we can assume that \(a_2\), which was defined in terms of \(T_1\), is continuous. 
If \(a_2\) ever vanishes then it does so on a dense set and so must vanish everywhere.
We repeat this argument, so without loss of generality all \(a_i > 0\) everywhere. 

We therefore have a smooth foliation with global coordinates \((x, \uu, \vv{i})\) as in Theorem \ref{thm:themetriconaV}.
Construct these coordinates by picking a \(p_0 \in \leaf{0}\) and \(p_0 = (0,0,\dots, 0)\). 
Since the orbit of 0 in the image of \(B\) is dense there exists a sequence \(g_k \in G \setminus \{e\}\) such that \(p_k := g_k(p_0) \in \leaf{k}\) and \(B(\leaf{k}) \to 0\) as \(k \to \infty\). 
Since \(G\) acts properly discontinuously on \(\tilde{M}\), there can only be finitely many \(p_k\) in any compact neighborhood of \(p_0\).
We will now show that the \(\uu\) and \(\vv{i}\) coordinates of the \(p_k\) must be bounded which finishes the proof.

To see this let \( x_k \) be the \( x \)-coordinate of \( p_k \) and let \(q_k = (x_k, 0, \dots, 0)\).
We want \( \beta \neq 0 \) at \( q_k \).
If taking such a subsequence is impossible then \( \beta(q_k) = \beta(x_k, 0, \dots, 0) = 0 \) for all sufficiently large \( k \).
In this case we may pick a different \( p_0 \in \leaf{0} \), say \( p_0 = (0, \varepsilon, 0, \dots, 0) \) for some small \( \varepsilon \) and constructing new coordinates with this new \( p_0 \) as the origin.
This will not change which leaves the \( p_k \) are in but will shift the \( q_k \) such that \( \beta(q_k) \neq 0 \) for all sufficiently large \( k \).

Then as \(k \to \infty\), \(q_k \to p_0\) since \( B(\mathcal{F}_{q_k}) = B(\mathcal{F}_{p_k}) \to 0 \) and hence \( x_k \to 0 \). 
Consider also \(g_k^{-1}(q_k)\), which lies in \(\leaf{0}\) because \(q_k\) and \(p_k\) are connected by a geodesic in the leaf \(\leaf{k}\), so \(g_k^{-1}(q_k)\) is connected to \(p_0\) by a geodesic in the leaf \(\leaf{0}\).
Furthermore, the length of this geodesic is also preserved.
Therefore if we show that the \( \uu \) and \( \vv{i} \) coordinates of \(g_k^{-1}(q_k)\) are bounded then the coordinates for \( p_k \) are bounded as well.

To see that the \(\uu\) coordinates of \(g_k^{-1}(q_k)\) must be bounded, let \(\uu_k\) be the \(\uu\)-coordinate of \(g_k^{-1}(q_k)\) and suppose it is not bounded.
So without loss of generality \(\uu_k \to \infty\) and is monotone increasing.
Recall that \(a_1 = \frac{f_1(x)}{\bfun(x,\uu)}\) and \(a_1\) is preserved by isometries,
\begin{equation*}
a_1(g_k^{-1}(q_k)) = a_1(q_k) \to a_1(p_0) = f_1(0) \quad \text{as } k \to \infty.
\end{equation*}
However, \(a_1(g_k^{-1}(q_k)) = \frac{f_1(0)}{\bfun(0, \uu_k)}\).
So as \(k \to \infty\), and \(\uu_k \to \infty\), we have \(\bfun(0,\uu_k) \to 1 = \bfun(0,0)\). 
But \( \bfun(0,\uu) \) is convex (and non-constant) since \( \bfun_{\uu \uu} = -\frac{\scal}{2}\bfun > 0 \) and hence the \( \uu \)-coordinates of \(g_k^{-1}(q_k)\) must be bounded.

For the \(\vv{i}\) coordinates we compute
\begin{equation*}
T_i(e_2(a_j)) = e_2(T_i(a_j)) + [T_i, e_2](a_j) = -(\nabla_{e_2}T_i)(a_j) = \begin{cases} a_1 a_j \beta & i = 1, \\ 0 & i > 1. \end{cases}
\end{equation*}
Let \(e_2^m(a_j) = e_2(e_2(\cdots e_2(a_j)))\) where we take the directional derivative in the \(e_2\) direction \(m\) times. 
Then for \(i \geq 2\):
\begin{align*}
T_i(e_2^2(a_j)) &= e_2(T_i(e_2(a_j))) + [T_i, e_2](e_2(a_j)) = 0 - \nabla_{e_2} T_i(e_2(a_j)) \\
&= a_i T_{i-1}(e_2(a_j)) - a_{i+1} T_{i+1}(e_2(a_j)) = \begin{cases} a_2 a_1 a_j \beta & i = 2, \\ 0 & i > 2. \end{cases}
\end{align*}
Repeating this argument for \(T_i(e_2^m(a_j))\), \(i \geq m\), we conclude that
\begin{equation*}
T_i(e_2^m(a_j)) = \begin{cases} a_m a_{m-1} \cdots a_2 a_1 a_j \beta & i = m, \\ 0 & i > m. \end{cases}
\end{equation*}
Note also that \(T_i(e_1(a_j)) = e_1(T_i(a_j)) + [T_i, e_1](a_j) = 0\).

Similar to before, let \(\vv{i,k}\) be the \(\vv{i}\) coordinate of \(g_k^{-1}(q_k)\). 
Suppose \(\vv{1,k}\) is not bounded, so without loss of generality and up to taking a subsequence \(\vv{1,k} \to \infty\) as \(k \to \infty\). 
Then the above computations give that the value of \(e_2(a_j)\) changes linearly in the \(\vv{1}\) direction and does not depend on \(\vv{2}, \dots, \vv{n}\).
Furthermore, this linear change is not constant since \( \beta(g_k^{-1}(q_k)) = \beta(q_k) \neq 0 \) and all \( a_i > 0 \).
In our case we also have \(x\) fixed and we know that the \(\uu_k\) are bounded.
Hence as we take \(\vv{1} \to \infty\), \(e_2(a_j)\) diverges.
But this is impossible as it must have the same value at \(g_k^{-1}(q_k)\) as it does at \(q_k\) and \(q_k \to p_0\), so it must approach the value of \(e_2(a_j)\) at \(p_0\) (up to possibly changing the sign). 
Hence \(\vv{1,k}\) must be bounded.

We now repeat this argument, looking at the value of \(e_2^m(a_j)\) and using that to show that \(\vv{m,k}\) is bounded since the former depends linearly on the latter. 
Furthermore, coordinates \( \vv{1,k}, \dots, \vv{m-1,k} \) are bounded and \( e_2^m(a_j) \) does not depend on \( \vv{m+1,k}, \dots, \vv{n,k} \).
Finally we conclude that all \(\vv{i,k}\) are bounded as well, so \(\pi_1(M)\) must act properly discontinuously on \(\R\) and so is either trivial or \(\Z\). 
\end{proof}

\printbibliography

\end{document}